\documentclass{amsart}
\usepackage[utf8]{inputenc}
\usepackage{hyperref}
\usepackage[numbers]{natbib}
\usepackage{amsfonts}
\usepackage{bbm}
\usepackage{amsmath}
\usepackage{mathtools}
\usepackage{caption}
\usepackage{subcaption}
\usepackage{graphicx}
\usepackage{xcolor}
\usepackage{amsthm}
\usepackage{geometry}
\newtheorem{theorem}{Theorem}[section]

\newtheorem{lemma}[theorem]{Lemma}
\graphicspath{ {./images/} }
\title{Tight Frames generated by a graph Short-time Fourier Transform}
\date{\today}

\author{Martin Buck}
\address{Department of Mathematics, Tufts University, Medford MA 02131, USA}
\email{martin.buck@tufts.edu}
\author{Kasso A.~Okoudjou}
\address{Department of Mathematics, Tufts University, Medford MA 02131, USA}
\email{kasso.okoudjou@tufts.edu}
\begin{document}

\begin{abstract}
A \textit{graph short-time Fourier transform} is defined using the eigenvectors of the graph Laplacian and a graph heat kernel as a window parametrized by a nonnegative time parameter $t$. We show that the corresponding Gabor-like system forms a frame for $\mathbb{C}^d$ and give a description of the spectrum of the corresponding frame operator in terms of the graph heat kernel and the spectrum of the underlying graph Laplacian. 
For two classes of algebraic graphs, we prove the frame is tight and independent of the window parameter $t$.
\end{abstract}

\maketitle
\section{Introduction and Background}

The short-time Fourier transform is an essential and widely used tool in Fourier analysis, quantum mechanics, and signal processing. Pioneered by J. von Neumann and Nobel prize winner Dennis Gabor \cite{gabor1946theory} in the 20th century, the short-time Fourier transform (STFT) provides joint time-frequency or space-frequency information of a function defined on $\mathbb{R}^d$. This is accomplished by localizing the standard Fourier transform via a \textit{window} function. As a result, the STFT combines features of both the function and its Fourier transform in a way that provides important local frequency information that traditional Fourier transform methods cannot provide. For example, in speech analysis and audio signal processing the STFT is used to detect pitch changes and transient events in the underlying signal \cite{smith2011spectral} \cite{griffin1984signal}. In image processing and optics, the STFT and associated Gabor system has been applied to denoising image patches, feature extraction, and coherent optical setups \cite{dawood2013efficient} \cite{bastiaans1998gabor}. Windowed
Fourier methods like the STFT and the related wavelet transforms are fundamental to communication
algorithms and to compression algorithms such as JPEG \cite{marcellin2000overview} \cite{baggett1990processing} \cite{galli2008recent}. In addition to its utility in application, the STFT draws from diverse theories in mathematics including complex analysis, partial differential equations, the Heisenberg group, representation theory, and operator algebras \cite{grochenig2001foundations} \cite{folland1989harmonic}.

The classical STFT and related Gabor transformation are defined when the signal of interest $f$ and the window $g$ belongs to dual function spaces, e.g.,  $f, g \in L^2(\mathbb{R}^d)$, or $f$ is a tempered distribution and $g$ is a Schwartz function defined on $\mathbb{R}^d$ \cite{grochenig2001foundations}. More generally, these transformations can be extended to signals defined in non-Euclidean spaces. 
%For example, pairs $(f,g)$ that lie in conjugate dual spaces $\big(L^p(\mathbb{R})^d), L^q(\mathbb{R}^d)\big)$, or the continuous functions with compact support and bounded Radon measures, or the Schwartz functions and tempered distributions \cite{grochenig2001foundations}. 
For example, see \cite{hansen2002gabor} \cite{fujita2017gabor} for the case of signals defined on a sphere.  One can therefore ask how to generalize the STFT for signals defined on less regular spaces. 
%This begs the question: \textit{in what other spaces can we successfully define and meaningfully extend notions from classical harmonic analysis such as the STFT}? 
In this paper, we revisit this question for signals defined on (finite) graphs by  defining a \textit{graph short-time Fourier transform} that is built on
\begin{enumerate}
    \item The eigenvectors $\{\phi_v\}_{v \in V}$ of the graph Laplacian $L$ to perform the graph Fourier transform
    \item The graph heat kernel $H_t \coloneqq e^{-tL}$ as a window to localize and translate the graph Fourier transform. 
\end{enumerate} Our approach is related to the vertex-frequency analysis introduced in \cite{shuman2016vertex} and will lead to the construction of a time-dependent family of finite Gabor-like frames for $\mathbb{C}^d$.  One of the major difficulties in defining windowed Fourier transform methods on irregular spaces such as graphs is in defining meaningful notions of graph harmonics and graph translation \cite{shuman2013emerging} \cite{gavili2017shift}. On $\mathbb{R}^d$, the eigenfunctions of the Laplacian $e^{2 \pi i x \xi}$ are the harmonics used in the Fourier transform and translation relies on an underlying group structure. %For example, a window $g(x) \in L^2(\mathbb{R})$ uses the group $(\mathbb{R}, +)$ to systematically move across the domain as $g(t-x)$.  
For signals defined on graphs, the approach developed in \cite{shuman2016vertex} is based on a notion of a graph Fourier transform through the spectral analysis of a graph Laplacian, and has found success in graph signal processing and data science applications \cite{shuman2013emerging}. The graph heat kernel $e^{-tL}$ then is a natural window to consider in defining a graph short-time Fourier transform as it is the semigroup associated to $L$ \cite{bakry2014analysis}. Since the graph heat kernel is a fundamental solution to the heat equation on a graph, it can be viewed as encoding a notion of translation via a convolution-type operator and relating one vertex to another via a random walker. For example, for two vertices $(v_i, v_j)$ on a graph with underlying group structure such as the ring graph, $e^{-tL}(v_i, v_j) = k_t(v_i-v_j)$, where $k_t(\cdot) = \sum_{i = 1}^N e^{- \lambda_i t} \phi_{\lambda_i}(\cdot)$ and $(\lambda_i, \phi_{\lambda_i})$ are an eigenvalue and eigenvector of the graph Laplacian \cite{terras1999fourier}. These two graph operators, $L$ and $H_t$, are well-studied and fundamental tools in spectral graph theory, graph signal processing, and data science. Indeed, spectral graph theory is based on the study of the spectrum of various graph Laplacians which reveal a dazzling array of properties of the geometry of the graph and dynamic processes defined on the graph. This includes isoperimetric and Cheeger inequalities, mixing time of Markov chains, and minimizing energies of Hamiltonian systems \cite{chung1997spectral} \cite{hein2007graph}. In graph signal processing, a notion of graph Fourier transform is defined via a change-of-basis matrix consisting of eigenvectors of the graph Laplacian. This allows an efficient basis representation of signals defined on the graph that is tailored to the structure of the underlying graph. In data science and computer vision, the heat kernel is used to characterize the shape of a mesh describing an underlying manifold or a graph \cite{sun2009concise}. This is due in part to the heat kernel being isometric invariant. Furthermore, a plethora of clustering techniques, graph neural networks, and PageRanks are based on the spectrum of the graph Laplacian and graph heat kernels \cite{von2007tutorial} \cite{kipf2016semi}.

In this paper, we used the heat kernel to define the graph-short-time Fourier transform on finite graphs. As a consequence, we obtain a time-dependent family of finitely many signals whose spanning properties we investigate. We note that \cite{shuman2016vertex} was the first to propose this definition, but their main results were stated for general window functions. In particular, they prove the resulting Gabor-like system forms a frame for $\mathbb{R}^d$ \cite[Theorem 3]{shuman2016vertex}. By contrast, our paper zeros in on the heat kernel as the window of the graph-short-time-Fourier transform and obtains the full description of the spectra of the frame operators associated to these Gabor-like systems. 

Our main contributions can be summarized as follows. In Section~\ref{sec: gstft}
we introduce our definition of a  short-time Fourier transform on a general graph  \eqref{eq: gstft} using the graph Laplacian and graph heat kernel. Subsequently, an explicit form for the inverse transform, the associated frame operator, and its spectrum are provided in Theorems \ref{thm:frame}, \ref{thm: inv}, and \ref{thm:eig}. In Section~\ref{sec: gstft_tight} on two classes of algebraic graphs, we show in Theorems~\ref{thm:vert} and~\ref{thm: srg} that the associated Gabor frame is tight and independent of the window parameter $t \in \mathbb{R}_{t \geq 0}$.

\section{The STFT and Gabor Frames on $\mathbb{C}^N$} \label{sec: STFT}
Here, we recall key definitions and concepts from finite-dimensional Fourier analysis focusing on the discrete Fourier transform, the discrete STFT, and associated Gabor frames for $\mathbb{C}^N$. The corresponding operators defined on graphs in Section \ref{sec: gstft} will coincide with those defined here when the underlying graph is the Cayley graph of $\mathbb{Z}/N\mathbb{Z}$, the \textit{circle graph} or \textit{ring graph} on $N$ verices. This is due to the eigenvector structure of the graph Laplacian on the circle graph whose entries are powers of roots of unity \cite{terras1999fourier}. This graph short-time Fourier transform and associated graph Gabor frame can therefore  be viewed as generalizing the STFT and Gabor frames for $\mathbb{C}^N$.

Given a vector $f \in \mathbb{C}^N$, the \textit{discrete Fourier transform (DFT)} $\mathcal{F}: \mathbb{C}^N \rightarrow \mathbb{C}^N$ is defined pointwise as \begin{equation*}
    \mathcal{F}f(m) = \hat{f}(m) = \tfrac{1}{\sqrt{N}}\sum_{n = 0}^Nf(n)e^{- 2 \pi i n m/n}, \ m=0, \ldots, N-1
\end{equation*}
The DFT can also be expressed in matrix notation. Let the \textit{Fourier matrix} $W_N \in \mathbb{C}^{N \times N}$ be defined as
\begin{equation*}
    W_N = (\tfrac{1}{\sqrt{N}}\omega^{-rs})_{r,s=0}^{N-1}, \ \omega = e^{2 \pi i/N}.
\end{equation*}
Then, the DFT can be expressed simply as the matrix-vector product $\hat{f}=W_Nf$. The DFT shares fundamental properties with the classical Fourier transform defined on functions in $L^2(\mathbb{R}^d)$ such as an inversion (reconstruction) formula, Parseval-Plancherel, and the Poisson summation formula. These can be derived from the fact that the normalized harmonics $\frac{1}{\sqrt{N}}e^{2 \pi i n (\cdot)/N}$ for $\ n = 0, \ldots, N-1$ form an orthonormal basis for $\mathbb{C}^N$:
\begin{equation*}
    f = \frac{1}{\sqrt{N}}\sum_{n = 0}^N \hat{f}(n)e^{2 \pi i (\cdot)/N}, \ f \in \mathbb{C}^N.
\end{equation*}
The short-time Fourier transform localizes the Fourier transform and provides information not only on constituent frequencies that make up a signal or function but also when or where those frequencies occur. This is accomplished via the composition of a translation operator and a modulation operator acting on a window function. The \textit{cyclic shift operator} $T:\mathbb{C}^N \rightarrow \mathbb{C}^N$ is given by
\begin{equation*}
    Tf = T\big(f(0), f(1), \ldots, f(N-1)\big)^T = \big(f(N-1), f(0), f(1), \ldots, f(N-2)\big)^T.
\end{equation*}
Then, for $k \in \{0,1, \ldots, N-1\}$ the \textit{translation operator} $T_k$ is given by
\begin{equation*}
    T_kf(n) := T^kf(n) = f(n-k), \ n=0,1, \ldots, N-1.
\end{equation*}
The \textit{modulation operator} $M_l:\mathbb{C}^N \rightarrow \mathbb{C}^N$ for $l = 0, 1, \ldots, N-1$ is given by pointwise multiplication with harmonics $e^{2 \pi i l(\cdot)/N }$,
\begin{equation*}
    M_lf = \big( e^{2 \pi i l 0/N}f(0), e^{2 \pi i l 1/N}f(1), \ldots, e^{2 \pi i l (N-1)/N} f(N-1)\big)
\end{equation*}
Translation operators are also called time-shift operators. Likewise, modulation operators are called frequency-shift operators and much of time-frequency analysis is based on the interplay between these two operators which are related by
\begin{equation*}
    \mathcal{F}M_l = T_l\mathcal{F}.
\end{equation*}
The time-frequency operator $\pi(k,l):\mathbb{C}^N \rightarrow \mathbb{C}^N$ takes $f \mapsto \pi(k,l)f = M_lT_kf$. The \textit{discrete short-time Fourier transform (DSTFT)} $V_g:\mathbb{C}^N \rightarrow \mathbb{C}^{N \times N}$ with respect to a \textit{window} $g \in \mathbb{C}^N \backslash \{0\}$ is then given by
\begin{equation} \label{eq: DSTFT}
    V_gf(k,l) = \langle f, \pi(k,l)g \rangle = \sum_{n=0}^{N-1}f(n) \Bar{g}(n-k)e^{-2 \pi i l n/N}.
\end{equation}
The set $(g, \Gamma) = \{\pi(k,l)g\}_{(k,l)\in \Gamma}$ for $\Gamma = \{0,1,\ldots,N-1\}^2$ is called the \textit{full Gabor system} generated by the window $g$ and the set $\Gamma$. The DSTFT is then an inner product between $f$ and an element of a Gabor system, $V_gf(k,l) = \langle f, \pi(k,l)g \rangle$. A series of computations shows that the following reconstruction formula holds for each  $f \in \mathbb{C}^N$:
\begin{equation*}
    f(n) = \frac{1}{N ||g||^2}\sum_{k=0}^{N-1}\sum_{l=0}^{N-1}V_{g}f(k,l) g(n-k)e^{-2 \pi i l n /N}
\end{equation*}
Equivalently we see that   the Gabor system $(g, \Gamma)$ is a \textit{tight frame} \cite{casazza2012finite}. Formally, the collection of elements $\Psi = \{\psi_0, \ldots, \psi_{N-1}\}$ forms a \textit{frame} for $\mathbb{C}^N$ if there exist constants $0<A\leq B$ such that for all $f \in \mathbb{C}^N$, we have
\begin{equation} \label{eq: frame_cond}
    A ||f||^2 \leq \sum_{i=0}^{N-1} |\langle f, \psi_i \rangle|^2 \leq B ||f||^2
\end{equation}
%Equivalently, $\Psi$ is a spanning set for $\mathbb{C}^N$. 
The optimal such constants $A,B$ are referred to as the frame bounds for $\Psi$. Frames are usually analyzed through the study of their frame operators. To define these operators, we first let the \textit{analysis operator} $C:\mathbb{C}^N \rightarrow \ell^2(\mathbb{C}^N)$ take an element to its sequence of frame coefficients, $Cf = \{\langle f, \psi_i \rangle \}_{i=0}^{N-1}$. The \textit{frame operator} $S:\mathbb{C}^N \rightarrow \mathbb{C}^N$ is then defined as the composition $C^*C$, so that
\begin{equation*}
    Sf = \sum_{i=0}^{N-1} \langle f, \psi_i \rangle f.
\end{equation*}
The frame condition~\eqref{eq: frame_cond} can then be re-written in terms of the frame operator as
\begin{equation*}
    A ||f||^2 \leq \langle Sf, f\rangle \leq B ||f||^2.
\end{equation*}
An especially important class of frames are tight frames, where the frame bounds coincide $A=B$. When the frame is tight, we have a remarkably compact representation of an element $f \in \mathbb{C}^N$ which resembles that of an orthonormal basis,
\begin{equation*}
    f = \sum_{i=0}^{N-1} \langle f, \psi_i \rangle \psi_i
\end{equation*}
due to the frame $\{\frac{1}{\sqrt{A}}\psi_i\}_{i=0}^{N-1}$ being self-dual.
%with the \textit{canonical dual-frame} $\{S^{-1}\psi_i\}_{i=0}^{N-1}$.
This implies that after this re-scaling, tight frames satisfy \textit{Parseval-Plancharel}:
\begin{equation*}
    ||f||^2 = \sum_{i=0}^{N-1} |\langle f, \psi_i \rangle|^2
\end{equation*}
Tight frames also are numerically stable with an optimal condition number, can be viewed as an orthogonal projection of an orthonormal basis from a larger Hilbert space, and are required to minimize the frame potential, $\sum _{j,j' =0}^{N-1} |\langle  e_j, e_{j'} \rangle|^2$ \cite{han2014operator} \cite{fickus2015detailing} \cite{benedetto2003finite}. In Section~\ref{sec: gstft_tight} we consider when the underlying graph structure determined if graph Gabor frames defined below are indeed tight. For further work on the fundamentals of frame theory and tight frames see \cite{casazza2012finite} \cite{heil2010basis} \cite{grochenig2001foundations}.

Part of the utility of the DSTFT is it provides joint time-frequency information of a function $f \in \mathbb{C}^N$. If $f(t) \in \mathbb{R}^N$ is a sampled piecewise cosine containing multiple constituent frequencies, the DSTFT aids in determining when the frequency of $f(t)$ changes in addition to identifying the constituent frequencies; a task the ordinary Fourier transform cannot accomplish, see Figures \ref{fig:ft} \ref{fig:stft}.
\begin{figure}[h]
    \centering
    \begin{subfigure}[b]{0.47\textwidth}
         \centering
         \includegraphics[width=\textwidth]{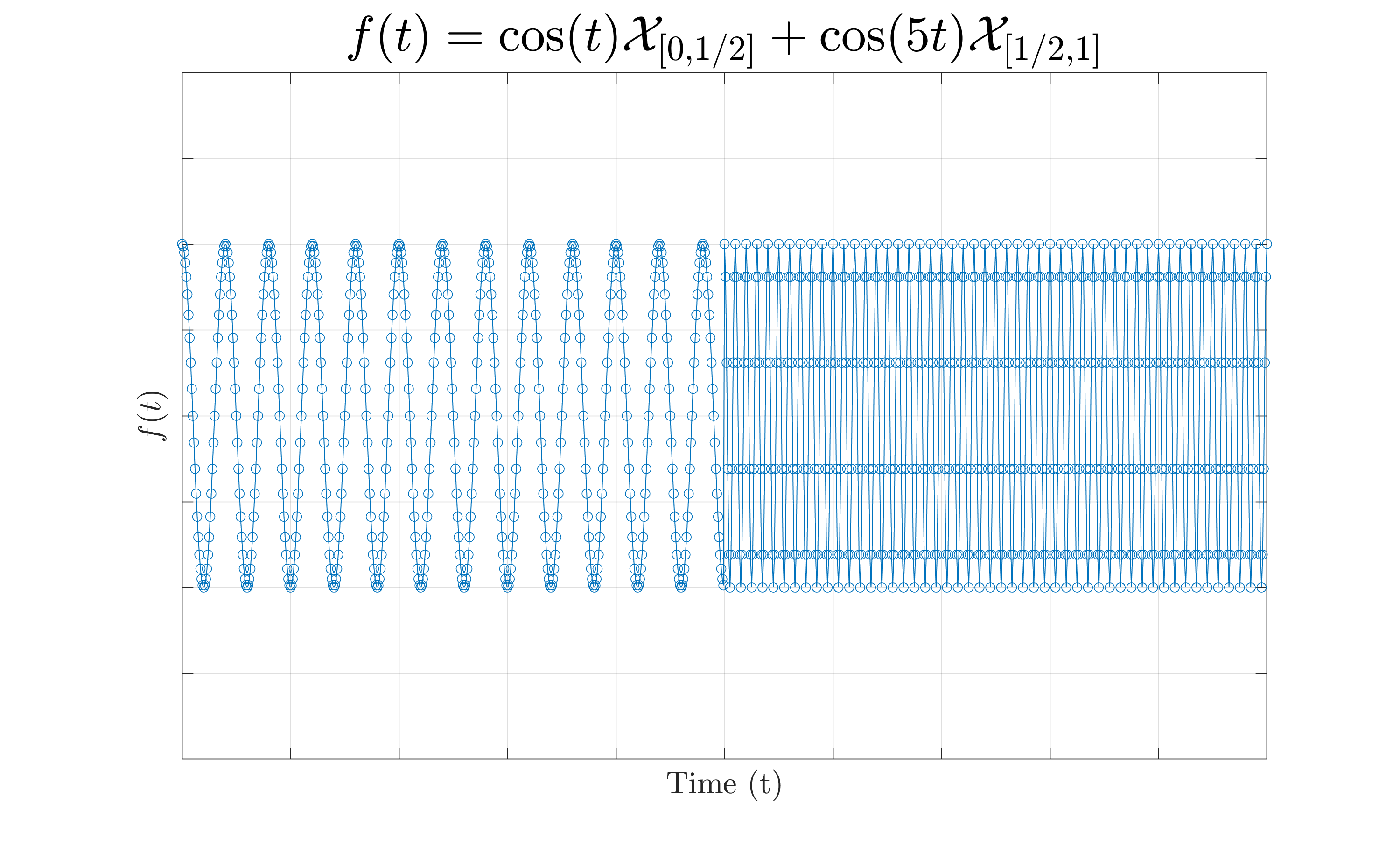}

         \caption{${f(n)= \cos(\alpha_1 n) \mathcal{X}_{[0,1/2]} + \cos(\alpha_2 n) \mathcal{X}_{[1/2,1]}}$}
         \label{fig:cos}
     \end{subfigure}
     \hspace{.5cm}
     \begin{subfigure}[b]{0.47\textwidth}
         \centering
         \includegraphics[width=\textwidth]{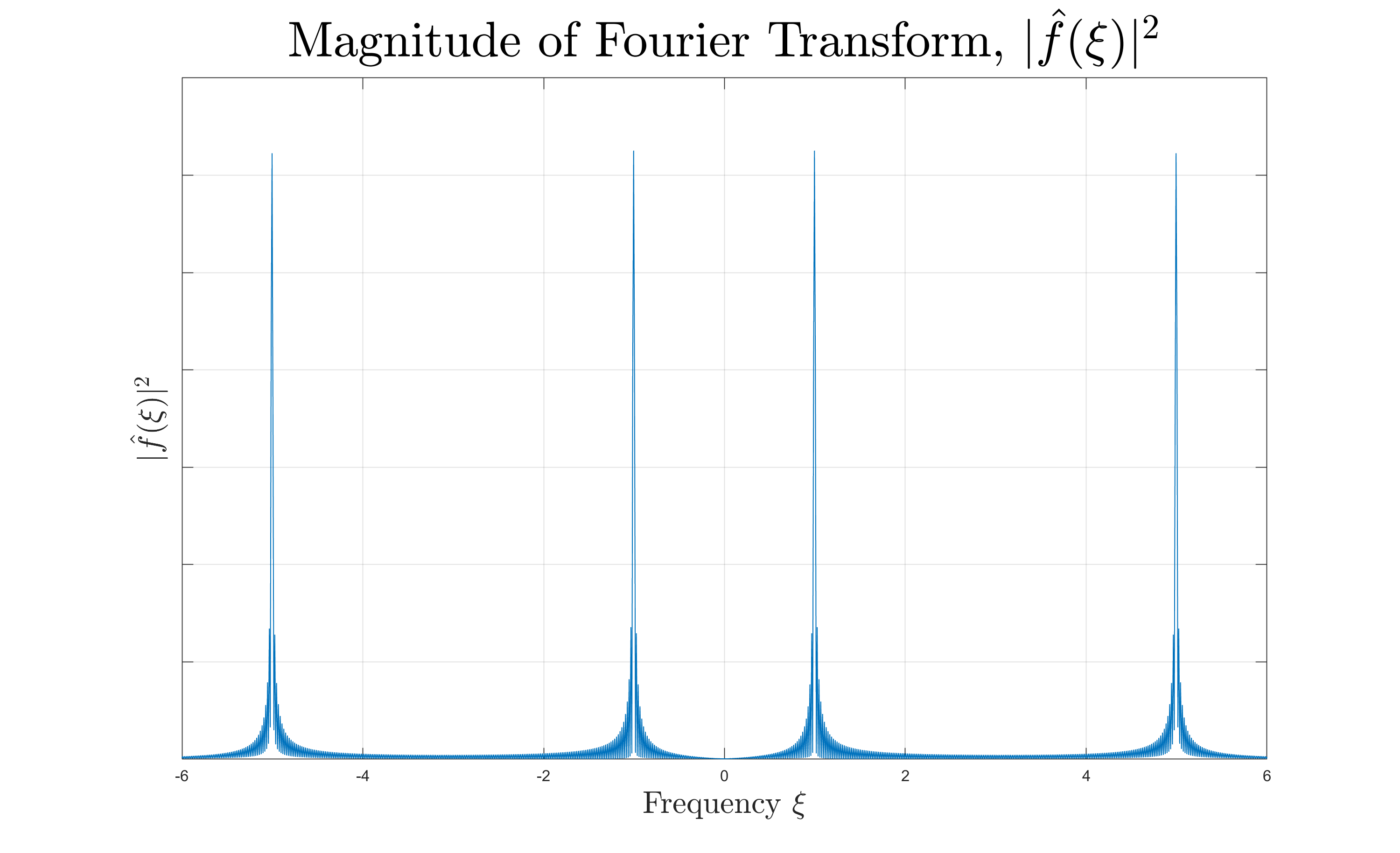}
         \caption{Magnitude of DFT $|\hat{f}(m)|^2$}
         \label{fig:ftcos}
     \end{subfigure}
     \caption{A piecewise cosine $f(n)$ and the magnitude of the DFT $|\hat{f}(m)|^2$. The magnitude $|\hat{f}(m)|^2$ localizes around the modulus of the constituent frequencies but provides no information on when the frequency change occurs.}
     \label{fig:ft}
\end{figure}
\begin{figure}[h]
    \centering
    \begin{subfigure}[b]{0.45\textwidth}
         \centering
         \includegraphics[width=\textwidth]{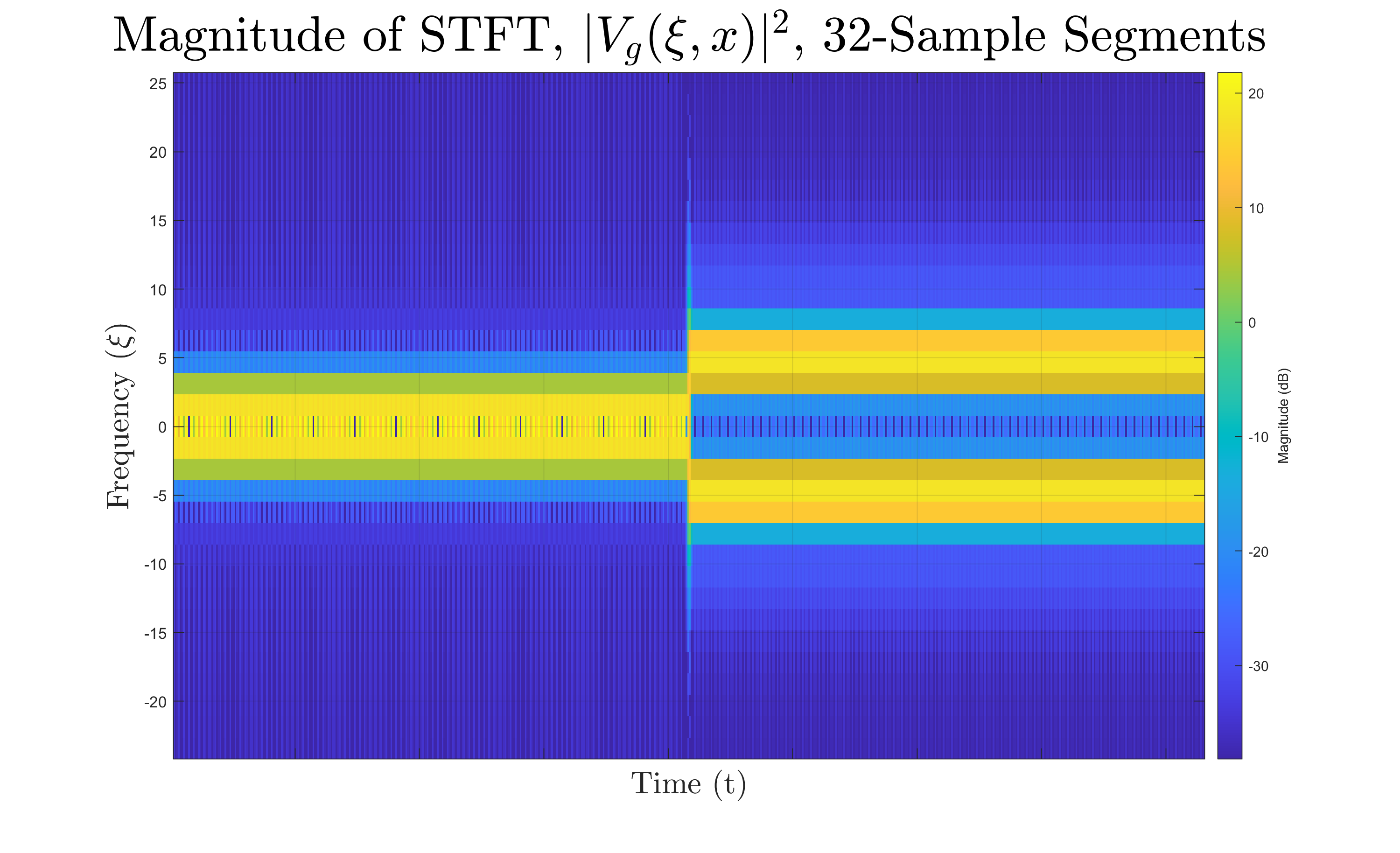}
         \caption{Spectrogram $|V_{g_1}f(k, l)|^2$}
         \label{fig:stftfcos1}
     \end{subfigure}
     \hspace{.5cm}
     \begin{subfigure}[b]{0.45\textwidth}
         \centering
         \includegraphics[width=\textwidth]{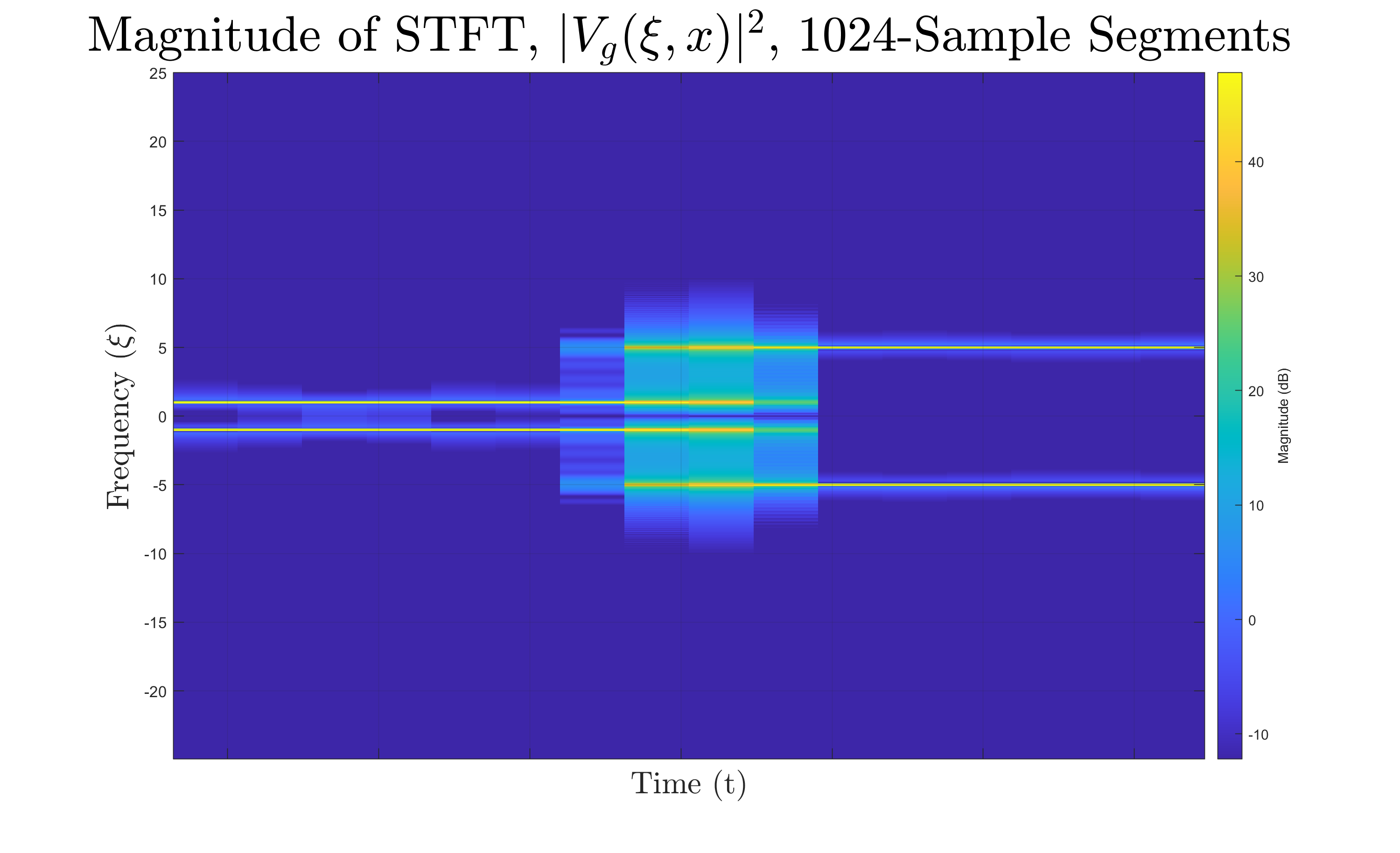}
         \caption{Spectrogram $|V_{g_2}f(k, l)|^2$}
         \label{fig:sftfcos2}
     \end{subfigure}
     \caption{Spectrograms visualize the magnitude of the DSTFT $|V_{g}f(k, l)|^2$ for two different sized windows. The windowing allows the DSTFT to determine when the frequency change occurs in $f(n)$ from Figure as well as the constituent frequencies \ref{fig:cos}.}
     \label{fig:stft}
\end{figure}

\section{A Short-time Fourier Transform and Gabor Frames on Graphs} \label{sec: gstft}

Let $V=\{v_1, \ldots, v_N\}$ and $E\subset \{(v_i, v_j): v_i, v_j \in V, i \not = j\}$ be the finite vertex and edge set of the graph $G=(V,E)$. We consider an undirected and connected graph with uniform edge weights equal to one. Let $\mathcal{C}(V)$ be the complex vector space of functions $f:V \rightarrow \mathbb{C}^N$ on the vertices of the graph. A standard measure to associate with $G$ is the counting measure. For a subset of vertices $V' \subset V$, the measure $\mu(V')$ is simply its cardinality:
\begin{equation*}
    \mu(V') = |V'|.
\end{equation*}
Given the measure $\mu$, the space $\mathcal{C}(V)$ inherits a Hilbert space structure with inner product
\begin{equation*}
        \langle f,g \rangle = \sum_{v \in V} f(v)\Bar{g}(v) \mu(v) = \sum_{v \in V} f(v)\Bar{g}(v).
\end{equation*}
The \textit{graph Laplacian} $L:\mathcal{C}(V) \rightarrow \mathbb{R}^N$ as an operator is defined as
\begin{equation*}
    Lf(v) = \sum_{u\sim v} \big(f(v)-f(u)\big).
\end{equation*}
In matrix notation, $L=D-A$, where $A\in \mathbb{R}^{N \times N}$ and $D\in \mathbb{R}^{N \times N}$ are the adjacency matrix and diagonal degree matrix associated with $G$. Recall, that the semigroup associated with the graph Laplacian is is the matrix exponential, called the \textit{graph heat kernel} ${H_t \in \mathbb{R}^{N \times N}}$:
\begin{equation}
    H_t \coloneqq e^{-tL} = \sum_{k=0}^ \infty \frac{(-t)^k}{k!} L^k.
\end{equation}
The graph Laplacian $L$ is also the \textit{generator} of the semigroup $e^{-tL}$ \cite{bakry2014analysis}. By the \textit{spectral theorem}, $L$ has orthonormal set of eigenvectors $\Phi=\{\phi_1, \phi_2, \ldots, \phi_N\}$ and real eigenvalues $\Lambda = \{\lambda_1, \lambda_2, ... \lambda_N\} \subset [0, \infty)$. Define a function $\Tilde{\phi}:(G,\Lambda) \rightarrow \mathbb{C}$ that takes $(v_i, \lambda_j)\mapsto \phi_{\lambda_j}(v_i)$. We will use the $\phi_{\lambda_j}(v_i)$ notation interchangeably with the $\Tilde{\phi}$. The \textbf{graph Fourier transform} $\hat{f}$ of a function $f:V \rightarrow \mathbb{C}^N$ on the vertices of $G$ is defined as the expansion of $f$ in terms of this orthonormal basis of eigenvectors,
\begin{equation*}
    \hat{f}(\lambda_l) \coloneqq \langle f, \phi_l \rangle = \sum_{i=1}^N f(v_i)\Bar{\phi}_{\lambda_l}(v_i).
\end{equation*}
This can also be expressed as the matrix multiplication $\hat{f}=\Phi^*f$, where now $\Phi\in \mathbb{C}^{N\times N}$ is the matrix with columns $\phi_1, \phi_2, \ldots, \phi_N$. Equivalently, $\Phi_{ij}$ are the function values $\phi_j(v_i)$. We therefore have the inversion formula and function representation
\begin{equation*}
    f(v_i) = \sum_{l=1}^N \hat{f}(\lambda_l) \phi_{\lambda_l}(v_i) = \Phi \hat{f}
\end{equation*}
This graph Fourier transform shares a number of key properties with the classical Fourier transform. It can be shown that the graph Fourier transform obeys the \textit{Parseval-Plancharel Identity} $\langle f, g \rangle = \langle \hat{f}, \hat{g} \rangle$ simply due to $\Phi$ being unitary:
\begin{equation*}
%\begin{split}
    \langle \hat{f}, \hat{g} \rangle  = \langle \Phi^*f, \Phi^*g \rangle = \langle \Phi\Phi^*f, g \rangle = \langle f, g \rangle
%\end{split}
\end{equation*}
Furthermore, the eigenvalues correspond to \textit{graph energies} in the sense that the quadratic form \begin{equation*}
    \bar{\phi}_{\lambda} L \phi_{\lambda} = \sum_{(u,v) \in E} (\phi_{\lambda}(u) - \phi_{\lambda}(v))^2
\end{equation*} is small when there is small variation of the eigenvector $\phi$ across the graph. When the underlying graph is sufficiently regular, the entries of the eigenvectors are roots of unity or samples of a cosine \cite{shuman2013emerging}. In fact, when the graph is a ring graph the graph Fourier transform is the \textit{Discrete Fourier Transform (DFT)} used in signal processing \cite{casazza2012finite}.\\

For the graph short-time Fourier transform, We set the window to be the columns of the graph heat kernel $h_t(v_i) \coloneqq H_t(\cdot, v_i)$. Each entry of the window $h_t(v_i) \geq 0$ is tunably localized around the vertex $v_i$, and we have underlying geometry of the graph encoded in $H_t$ via decay estimates on the trace in large time or recovering geodisics (shortest-path distances) in short time. Because the graph Laplacian is diagonalizable, $H_t$ has the compact form $H_t = \Phi e^{-t \Lambda} \Phi^*$, which reads entry-wise as
\begin{equation*}
    H_t(v_i, v_j) = \sum_{l=1}^N e^{-\lambda_l t} \phi_{\lambda_l}(v_i) \Bar{\phi}_{\lambda_l}(v_j).
\end{equation*}
It is not difficult show that the entries of $H_t$ are positive for $t>0$ and compute an explicit form for the norms of its columns. We show in Theorem \ref{thm:eig} that the spectrum of the frame operator corresponds to these column norms. Define the \textbf{graph short-time Fourier transform (GSTFT)} $V_t: \mathbb{C}^N \rightarrow \mathbb{C}^{N \times N}$ of a function $f:V \rightarrow \mathbb{C}^N$ on the vertices of the graph $G$:
\begin{equation} \label{eq: gstft}
    (V_t f)(v_i, \lambda_j) \coloneqq \sum_{k=1}^N f(v_k)H_t(v_i, v_k) \Bar{\phi}_{\lambda_j}(v_k).
\end{equation}
This definition can be viewed as an extension of a discrete short-time Fourier transform on $\mathbb{C}^N$ in the same sense that the graph Fourier transform extends the discrete Fourier transform \cite{casazza2012finite}. It is a representation of a function on the the graph with entries indexed by vertex $v_i$ and eigenvalue $\lambda_j$ for $i,j \in \{1,2, \ldots, N\}$. We can get a simpler representation of this GSTFT in terms of an inner product by defining an associated graph Gabor system. Let $D_i(t) \in \mathbb{R}^{N \times N}$ be the diagonal matrix containing the $i$-th column of the graph heat kernel matrix:
\begin{equation*}
    D_i(t) = \text{diag}(H_t(\cdot, v_i)).
\end{equation*}
Then, the \textit{graph Gabor system} associated with the GSTFT is
\begin{equation*}
    \{\psi_{ij}(t)\}_{i,j=1}^N \coloneqq \{D_i(t)\phi_{\lambda_{j}}\}_{i,j=1}^N.
\end{equation*}
Thus, we can write each entry of the GSTFT $V_t f(v_i, \lambda_j)$ as the inner product
\begin{equation*}
   (V_t f)(v_i, \lambda_j) = \langle f, \psi_{ij}(t) \rangle.
\end{equation*}
With the classical short-time Fourier transform and associated Gabor system, the function can be reconstructed from its sampling on the time-frequency lattice when the Gabor system forms a frame. To show this graph Gabor system forms a frame for $\mathbb{C}^N$, we introduce the associated operators. The \textit{analysis operator} $A(t) \in \mathbb{C}^{N^2 \times N}$ is the matrix with rows that are the conjugate transpose graph Gabor atoms $\{\bar{\psi}(t)_{ij}\}_{i,j=1}^N$. The corresponding \textit{frame operator} $S(t) \in \mathbb{C}^{N \times N}$ is the composition $S(t) \coloneqq A(t)^*A(t)$ and the \textit{Grammian} $G(t) \in \mathbb{C}^{N^2 \times N^2}$ is the composition $G(t) \coloneqq A(t) A(t)^*$. We now show that the frame operator $S(t)$ is a diagonal matrix, explicitly compute its form, and show the set of graph Gabor system always forms a frame.

\begin{theorem} \label{thm:frame} The set of graph Gabor atoms $\{\psi_{ij}(t)\}_{i,j=1}^N$ forms a frame for $\mathbb{C}^N$ for all $t \geq 0$. Furthermore, the frame operator $S(t)$ is a diagonal matrix and has the following form:
    \begin{equation*}
        S(t) = \sum_{i=1}^N D_i^2(t)
    \end{equation*}
\end{theorem}

\begin{proof} Fix $i \in \{1,2, \ldots, N\}$. Then consider the $N \times N$ matrix $D_i(t)\Phi$. This matrix has $N$ columns that are a subset of the $N^2$ set of Gabor atoms. This matrix is the composition of a full-rank diagonal matrix and a unitary matrix and hence spans $\mathbb{C}^N$. Now, a column of $A_t^*$ is $D_i(t)\phi_{\lambda_j}$ and a row of $A_t$ is $\bar{\phi}_{\lambda_j}D_i(t)$ for $i,j \in \{1, 2, \ldots N\}$. We write the frame operator as a sum of $N^2$ outer products:
    \begin{align*}
        %\begin{split}
            S(t) &= \sum_{i=1}^N \sum_{j=1}^N D_i(t)\phi_{\lambda_j} \bar{\phi}_{\lambda_j} D_i(t)  = \sum_{i=1}^N D_i(t) \big[\sum_{j=1}^N \phi_{\lambda_j} \bar{\phi}_{\lambda_j} \big] D_i(t) \\
            &= \sum_{i=1}^N D_i(t) I_{N \times N} D_i(t) = \sum_{i=1}^N D_i^2(t)
        %\end{split}
    \end{align*}
\end{proof}
We can leverage this explicit for of the frame operator to explicitly characterize its spectrum and describe when the frame formed by the graph Gabor atoms are a tight frame.
\begin{theorem} \label{thm:eig} The eigenvalues $\gamma_j(t)$ for $j \in \{1, \ldots, N\}$ of the frame operator $S(t)$ are the norm-squared of the corresponding columns of the graph heat kernel, 
\begin{equation*}
    \gamma_j(t) = ||H_t(\cdot, v_j)||^2 = \sum_{l=1}^{N} e^{-2\lambda_lt}|\phi_l(v_j)|^2 
\end{equation*}
In particular, the eigenvalues are positive and the frame for $\mathbb{C}^N$ formed by the graph Gabor system $\{\psi_{ij}(t)\}_{i,j=1}^N$ is tight for times $t\geq 0$ when the columns of the graph heat kernel all have the same norm.
\end{theorem}
\begin{proof}
By Theorem \ref{thm:frame}, we know that the frame operator is a diagonal matrix and thus has eigenvalues given by the diagonal entries with eigenvectors as the standard basis vectors. By direct computation, these eigenvalues $\gamma_j(t)$ for $j \in \{1, \ldots, N\}$ have the form
\begin{equation*}
    \gamma_j(t) = S_{jj}(t) = \sum_{i=1}^N H_t(v_i, v_j)^2 = ||H_t(\cdot, v_j)||^2.
\end{equation*} 
The frame $\{\psi_{ij}(t)\}_{i,j=1}^N$ is tight when the eigenvalues of $S(t)$ are equal; that is, for times $t \in \mathbb{R}_{\geq 0}$ when $||H_t(\cdot, v_i)||^2=||H_t(\cdot, v_j)||^2$ for all $i,j \in \{1, 2, \ldots N\}$. To obtain an explicit formula for these eigenvalues, we compute:
\begin{equation*}
    \begin{split}
        \langle H_t(\cdot, v_i), H_t(\cdot, v_i) \rangle  &= \sum_{k=1}^N H_t(v_k, v_i)\Bar{H}_t(v_k, v_i) \\
        &= \sum_{k=1}^N \Big[ \sum_{l=1}^N e^{-\lambda_l t} \phi_{\lambda_l}(v_k) \Bar{\phi}_{\lambda_l}(v_i) \Big] \Big[ \sum_{l'=1}^N e^{-\lambda_{l'} t} \Bar{\phi}_{\lambda_{l'}}(v_k) \phi_{\lambda_{l'}}(v_i) \Big] \\
        &= \sum_{k=1}^N \Big[ \sum_{l=1}^N \sum_{l'=1}^N e^{-(\lambda_l+\lambda_{l'})t} \phi_{\lambda_l}(v_k) \Bar{\phi}_{\lambda_{l'}}(v_k) \phi_{\lambda_{l'}}(v_i) \Bar{\phi}_{\lambda_l}(v_i) \Big] \\
        &= \sum_{l=1}^N \sum_{l'=1}^N e^{-(\lambda_l + \lambda_{l'})t} \phi_{\lambda_{l'}}(v_i) \Bar{\phi}_{\lambda_l}(v_i) \Big[ \sum_{k=1}^N \phi_{\lambda_l}(v_k) \Bar{\phi}_{\lambda_{l'}}(v_k) \Big] \\
        &= \sum_{l=1}^N \sum_{l'=1}^N e^{-(\lambda_l + \lambda_{l'})t} \phi_{\lambda_{l'}}(v_i) \Bar{\phi}_{\lambda_l}(v_i) \delta_{ll'} \\
        &= \sum_{l=1}^N e^{-2\lambda_l t}|\phi_{\lambda_l}(v_i)|^2
    \end{split}
\end{equation*}
If $|\phi_{\lambda_l}(v_i)|^2 = 0, \text{ } \forall l \in \{1, 2, \ldots, N\}$, then $\{\phi_{\lambda_l}\}_{l=1}^N$ would not be a basis for $\mathbb{C}^N$. Therefore, $|\phi_{\lambda_l}(v_i)|^2 > 0$ for some $ l \in \{1, 2, \ldots, N\}$ and so $\sum_{l=1}^N e^{-2\lambda_l t}|\phi_{\lambda_l}(v_i)|^2 > 0$.

\end{proof}

In the limit as $t \rightarrow \infty$, the graph Gabor frame evolves towards a tight frame at a rate which is dominated by size of second smallest eigenvalue or \textit{Fiedler value} $\lambda_2$ of the graph Laplacian. This is due to the norm-squared of the columns of the graph heat kernel decaying at a rate that depends on the quantity $e^{-2 \lambda_l t}$ for $l=1, \ldots, N$, as seen in the above theorem. Since the graphs we are considering are connected, $\lambda_0=0$ with multiplicity $1$ and thus $\lim_{t \rightarrow \infty} ||H_t(\cdot, v_i)||^2 = 1/\sqrt{N}$, for all $v_i \in V$. In Figure \ref{fig: eig_diff}, a collection of random regular graphs are generated using the \textit{pairing model} (see \cite{bollobas1998random} and \cite{mckay1990uniform}) and the eigenvalue gap $|\gamma_{max} - \gamma_{min}|$ of the frame operator $S(t)$ is plotted as a function of $t$. We observe that the graph Gabor frame becomes tight at a faster rate for graphs with a larger Fiedler value.

\begin{figure}[htbp]
    \centering
    \includegraphics[scale=.16]{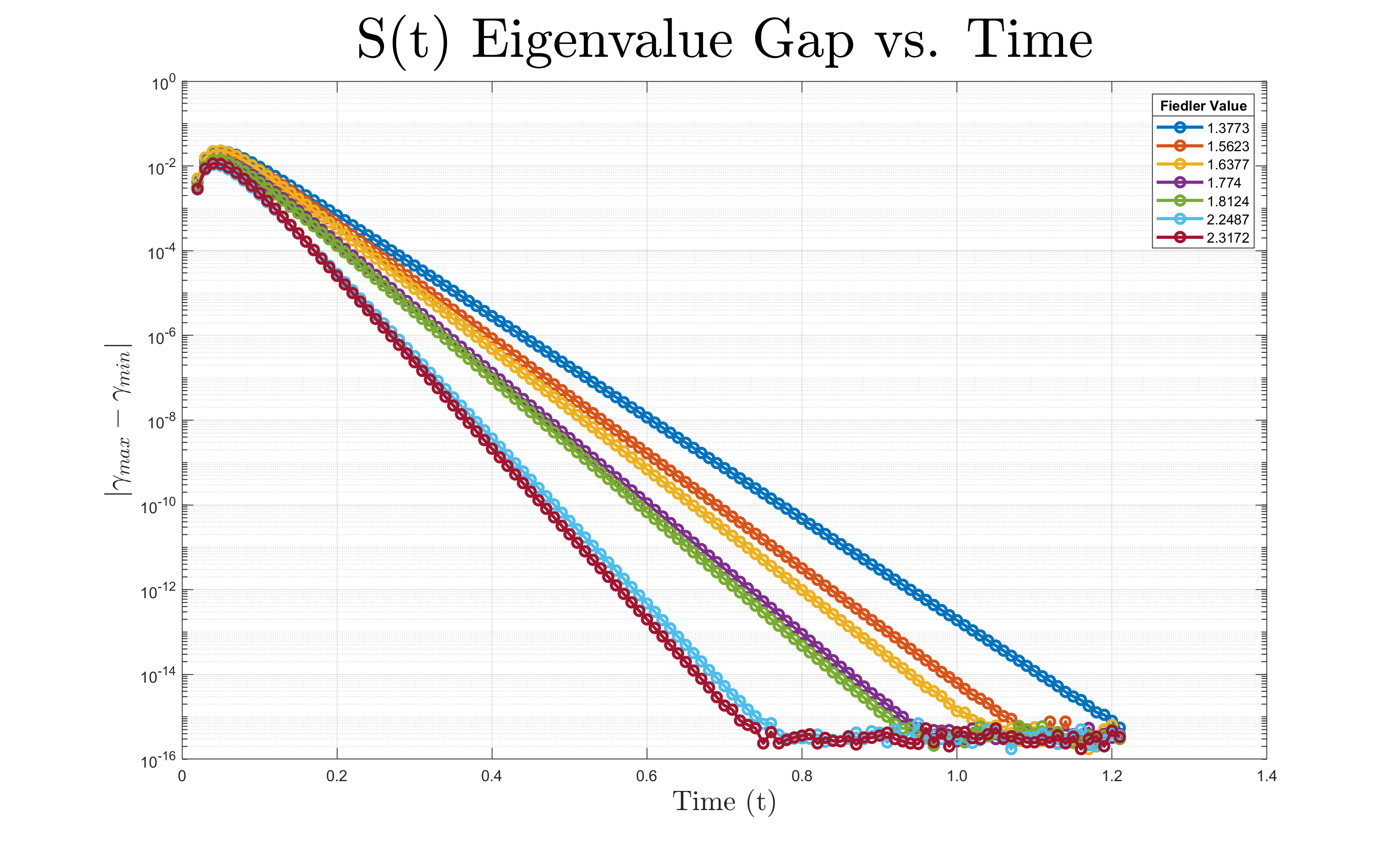}
    \caption{Difference between the maximum and minimum eigenvalue of the graph Gabor frame operator $S(t)$ as a function of time. Regular graphs with larger Fiedler values see faster decay towards a tight frame.}
    \label{fig: eig_diff}
\end{figure}

Like reconstruction in the case for $\mathbb{C}^N$, since the graph Gabor system forms a frame, we can reconstruct a function on the graph from its GSTFT.
\begin{theorem} \label{thm: inv}
     The operator $W_t: \mathbb{C}^{N \times N} \rightarrow \mathbb{C}^N$ defined as \begin{equation*}
    (W_tF)(v_i) = \frac{1}{||H_t(\cdot, v_i)||^2}\sum_{j=1}^N \phi_j(v_i) \big[ \sum_{k=1}^N F(v_k, \lambda_j) H_t(v_k, v_i) \big]
    \end{equation*} is a left inverse for the GSTFT $V_t$.
\end{theorem}

\begin{proof} By direct computation:
\begin{equation*}
    \begin{split}
        (W_t V_tf)(v_i) &= \frac{1}{||H_t(\cdot, v_i)||^2}\sum_{j=1}^N \phi_j(v_i) \big[ \sum_{k=1}^N V_tf(v_k, \lambda_j) H_t(v_k, v_i) \big] \\
        &= \frac{1}{||H_t(\cdot, v_i)||^2}\sum_{j=1}^N \phi_j(v_i) \big[ \sum_{k=1}^N [\sum_{l=1}^N f(v_l) \bar{\phi}_j(v_l) H_t(v_l, v_k)] H_t(v_k, v_i) \big] \\
        &= \frac{1}{||H_t(\cdot, v_i)||^2}\sum_{j=1}^N \sum_{k=1}^N \sum_{l=1}^N f(v_l) \phi_j(v_i) \bar{\phi}_j(v_l) H_t(v_l, v_k) H_t(v_k, v_i) \\
        &= \frac{1}{||H_t(\cdot, v_i)||^2}\sum_{k=1}^N \sum_{l=1}^N f(v_l) H_t(v_l, v_k) H_t(v_k, v_i) 
        \underbrace{[\sum_{j=1}^N \phi_j(v_i) \bar{\phi}_j(v_l)]}_{\delta_{il}} \\
        &= \frac{1}{||H_t(\cdot, v_i)||^2}\sum_{k=1}^N f(v_i) H_t(v_i, v_k) H_t(v_i, v_k) \\
        &= f(v_i) \frac{1}{||H_t(\cdot, v_i)||^2} \underbrace{\sum_{k=1}^N H_t(v_i, v_k)^2}_{||H_t(\cdot, v_i)||^2} \\
        &= f(v_i)
    \end{split}
\end{equation*}
\end{proof}

Note that this definition based on the interplay between the graph Laplacian $L$ and the graph heat kernel $H_t$ is remarkably tied to a notion of graph short-time Fourier transforms defined via a \textit{graph modulation} and \textit{graph translation} operator in \textit{Vertex-Frequency Analysis on Graphs} by Shuman et al. \cite{shuman2016vertex}. Recall that translation of a function on the graph $T_{i}f(v_k)$ from vertex $v_k$ to vertex $v_{i}$ for $i, k \in \{1, \ldots, N\}$ was a kernelized operation (a convolution) defined in the graph Fourier domain as
\begin{equation*}
    (T_{i}f)(v_k) \coloneqq \sqrt{N}(f*\delta_{i})(v_k) = \sqrt{N}\sum_{j=1}^{N} \hat{f}(\lambda_j)\bar{\phi}_{\lambda_j}(v_{i}) \phi_{\lambda_j}(v_k).
\end{equation*}
Further, modulation $(M_jf)(v_i)$ for $j \in \{1, \ldots, N\}$ was defined via pointwise multiplication with a eigenvector $\phi_j$,
\begin{equation*}
    (M_jf)(v_k) \coloneqq \sqrt{N} f(v_k) \phi_{\lambda_j}(v_k).
\end{equation*}
Their \textit{windowed Fourier transform} is then defined via the composition of these operators as in the classical case of the STFT. For a window function $g\in \mathbb{R}^N$, they define a \textit{windowed graph Fourier atom} by:
\begin{equation*}
    g_{ij}(v_k) \coloneqq (M_jT_ig)(v_k) = \underbrace{\sqrt{N} \phi_{\lambda_j}(v_k)}_{\text{Modulation}} \underbrace{\sqrt{N}\sum_{l=1}^{N}\hat{g}(\lambda_l)\bar{\phi}_{\lambda_l}(v_i)\phi_{\lambda_l}(v_k)}_{\text{Translation}}.
\end{equation*}
Define a window in the spectral domain as $\hat{g}(\lambda_l)=Ce^{-\tau \lambda_l}$ for a fixed $\tau > 0$, where the constant $C$ is chosen so that $\|g\|=1$. This is an example that is studied in \cite{shuman2016vertex} on the highly regular \textit{Minnesota road graph} dataset \cite{nr}. The window is defined in the graph Fourier domain since the translation operator is a kernelized operator given by integrating against $\hat{g}$. The $(i,j)$ entry of this windowed Fourier transform of a function $f:V \rightarrow \mathbb{R}^N$ is then given by:
\begin{equation} \label{eq: shuman_gstft}
    Sf(v_i,\lambda_j) \coloneqq \langle f, g_{ij}\rangle = N \sum_{k=1}^N f(v_k) \phi_{\lambda_j}(v_k) \big[\sum_{l=1}^{N}\hat{g}(\lambda_l)\bar{\phi}_{\lambda_l}(v_i)\phi_{\lambda_l}(v_k)\big].
\end{equation}
If the window is chosen to be this window $\hat{g}(\lambda_l)=Ce^{- \tau \lambda_l}$, then the above is written as:
\begin{equation*}
    Sf(v_i,\lambda_j) \coloneqq \langle f, g_{ij}\rangle = N \sum_{k=1}^N f(v_k) \big[\sum_{l=1}^{N} Ce^{- \tau \lambda_l} \phi_{\lambda_l}(v_i) \bar{\phi}_{\lambda_l}(v_k)\big] \bar{\phi}_{\lambda_j}(v_k).
\end{equation*}
Comparing with the GSTFT defined here in \ref{eq: gstft}, we can see the two definitions coincide for functions $f:V \rightarrow \mathbb{R}^N$ by unpacking the definition of the graph heat kernel $H_t$:
\begin{equation*}
\begin{split}
    (V_t f)(v_i, \lambda_j) & \coloneqq \sum_{k=1}^N f(v_k)H_t(v_i, v_k) \Bar{\phi}_{\lambda_j}(v_k) \\
    &= \sum_{k=1}^N f(v_k) \big[ \sum_{l=1}^N e^{-\lambda_l t} \phi_{\lambda_{l}}(v_i) \Bar{\phi}_{\lambda_l}(v_k)\big] \Bar{\phi}_{\lambda_j}(v_k)
\end{split}
\end{equation*}

\section{Tight Frames on Vertex-transitive and Strongly Regular Graphs} \label{sec: gstft_tight}
Given that the graph Gabor system is a frame and parametrized by $t\in \mathbb{R}_{\geq 0}$, a natural question to ask is for what values of $t$ is the frame tight? As discussed in \ref{sec: STFT}, tight frames have remarkable reconstruction properties and closely resemble orthonormal bases. Here we show that on two classes of algebraic graphs, the corresponding graph Gabor frame is tight for all $t \in \mathbb{R}_{\geq 0}$. We recall the definition of a graph automorphism which define \textit{vertex-transitive graphs} and \textit{strongly regular graphs}. As in the previous section, the graph $G=(V,E)$ is finite, undirected and unweighted, and assumed to be connected.

An \textit{automorphism} of a graph $G$ is a permutation $\pi$ of the vertices $V$ with the property that $(u,v) \in E$ if and only if $(\pi(u), \pi(v)) \in E$. The set of all automorphisms $\Pi$ of a graph together with the operation of composition $\text{Aut}(G) = (\Pi, \circ)$ is a group called the \textit{automorphism group}. A graph $G$ is called \textit{vertex-transitive} if $\text{Aut}(G)$ acts transitively on $V$. That is, for all pairs of vertices $u,v \in V$ there exists a permutation $\pi\in \text{Aut}(G)$ such that $\pi(u)=v$. This can also be understood in terms of group actions of $\text{Aut}(G)$. A graph $G$ on $|V|=n$ vertices that is neither complete nor empty is said to be \textit{strongly-regular} with parameters $(n,k,a,c)$ if 
\begin{enumerate}
    \item It is $k$-regular
    \item Every pair of adjacent vertices has $a$ common neighbors
    \item Every pair of non-adjacent vertices has $c$ common neighbors
\end{enumerate}
\begin{figure}[htpb]
    \centering
    \begin{subfigure}[b]{0.35\textwidth}
         \centering
         \includegraphics[width=\textwidth]{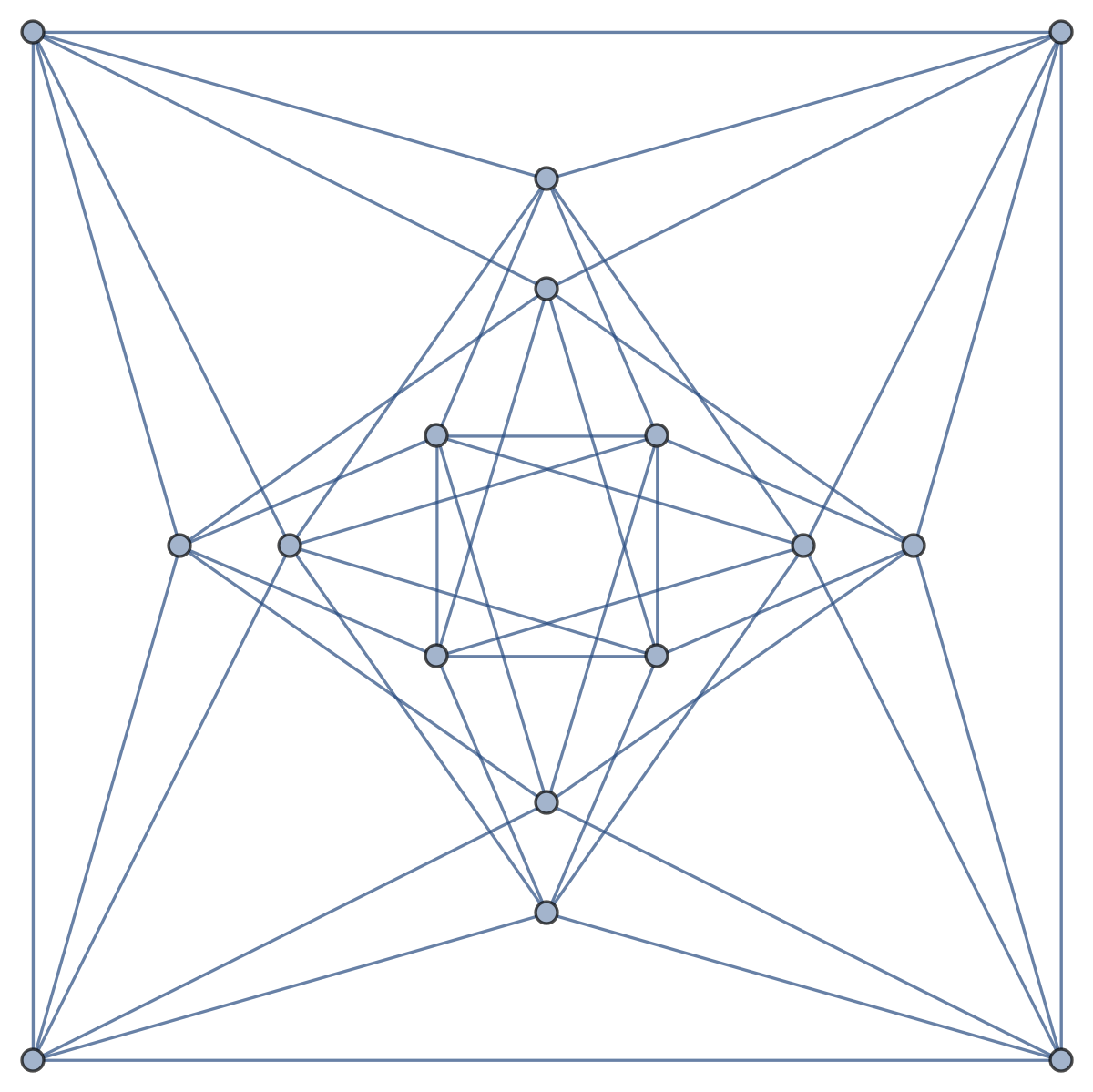}
         \caption{Shrikande graph}
         \label{fig: shrikande}
     \end{subfigure}
          \hspace{1cm}
     \begin{subfigure}[b]{0.35\textwidth}
         \centering
         \includegraphics[width=\textwidth]{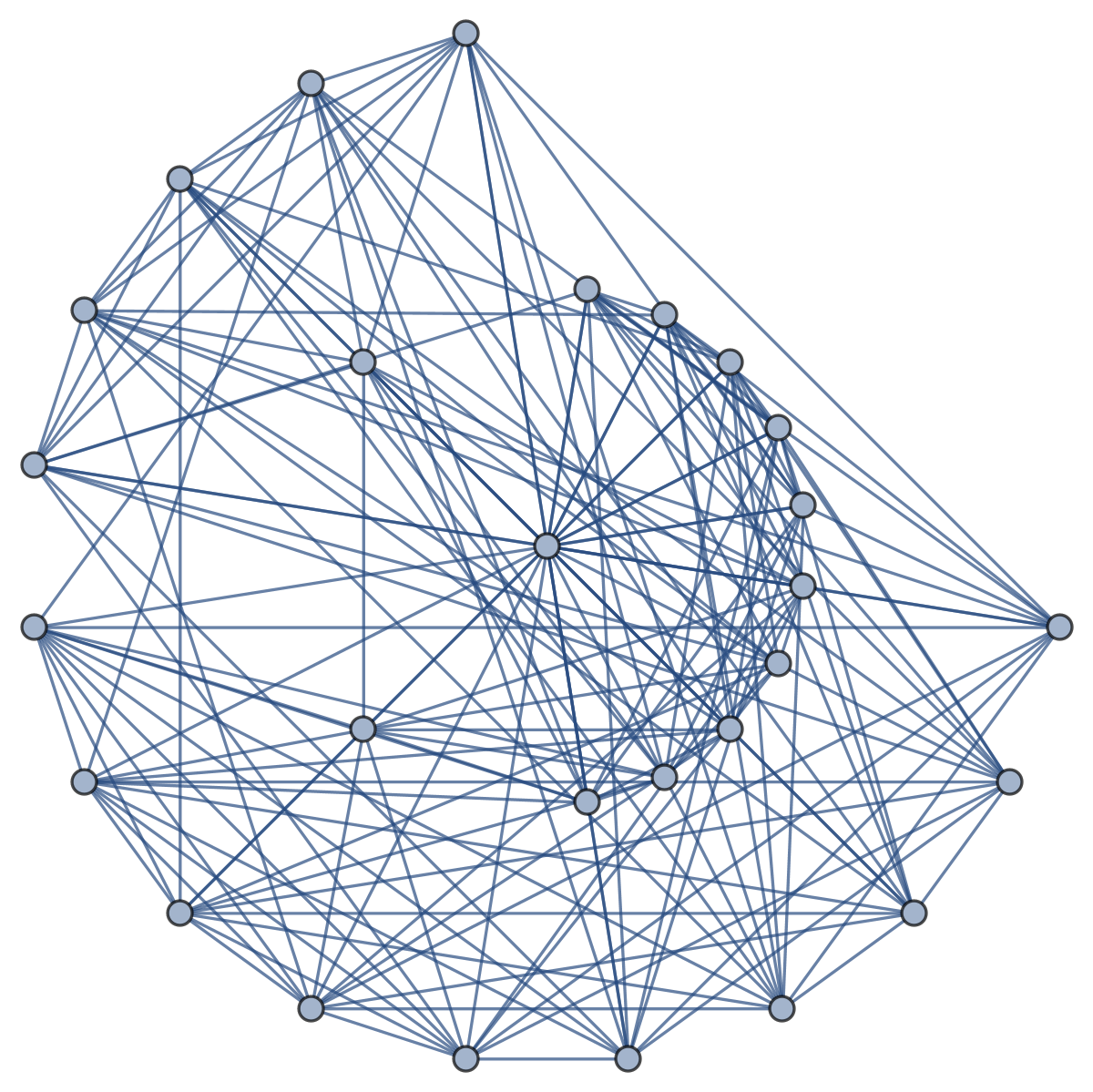}
         \caption{Chang graph}
         \label{fig: chang}
     \end{subfigure}
     \caption{The \textit{Shrikande graph} is a vertex-transitive and strongly regular graph with parameters $(16,6,2,2)$; it is the Cayley graph of $\mathbb{Z}_4 \times \mathbb{Z}_4$. The \textit{Chang graphs} are a family of strongly regular graphs with parameters $(28,12,6,4)$}.
     \label{fig: alg_graphs}
\end{figure}

\subsection{Vertex-transitive Graphs}
A consequence of a graph being vertex-transitive is that the adjacency matrix commutes with permutation matrices representing the maps $\pi \in \text{Aut}(G)$. A permutation $\pi$ of the vertex set $V$ can always be represented by a permutation matrix $P\in \mathbb{R}^{n \times n}$, where $P_{ij}=1$ if $v_i = \pi(v_j)$ and $P_{ij}=0$ otherwise. 

\begin{lemma} \label{biggs} (Biggs \cite[Chap. 15]{biggs1993algebraic})
    Let A be the adjacency matrix of a graph $G$ and $\pi$ a permutation of $V$. Then $\pi$ is an automorphism of $G$ if and only if $PA=AP$, where $P$ is the permutation matrix representing $\pi$.
\end{lemma}
A column of the graph heat kernel $h_i(t)$ can be interpreted as solutions to the graph heat equation with initial conditions of a dirac mass at the corresponding vertex:
\begin{equation*}
    \begin{aligned}
        \frac{d}{dt}h_i(t) = Lh_i(t) \\
        h_i(0) = \delta_{ij}
    \end{aligned}
\end{equation*}
This, together with Proposition \ref{biggs} and Theorem \ref{thm:eig} leads to the next theorem. The intuition is the following: one can map every vertex to another via a permutation such that the resulting graph is isomorphic to the original. Therefore, heat flow out of a source vertex can be mapped to heat flow out of another source vertex. Since the columns of the graph heat kernel describe such heat flow, each column will be a permutation of another and therefore have the same norm.
\begin{theorem} \label{thm:vert} The graph Gabor frame $\{\psi_{ij}(t)\}_{i,j=1}^N$ forms a tight frame for every time $t\in [0, \infty)$ when the underlying graph is vertex-transitive.
\end{theorem}
\begin{proof}
    Let $\pi \in \text{Aut}(G)$ be the automorphism such that $\pi(v_i)=v_j$. Then the associated permutation matrix $P$ has the property that $Pe_i=e_j$. Thus:
\begin{equation*}
     h_j = H_te_j = H_t(Pe_i) = (H_tP)e_i
\end{equation*}
Now, if the graph heat kernel $H_t$ and the permutation matrix $P$ commute then we are done as $h_j = (H_t P) e_i = (PH_t)e_i =P h_i$. The graph $G$ is assumed to be vertex-transitive and hence $PA=AP$. Therefore, since $H_t=\sum_{l=0}^{\infty} \frac{(-t)^l}{l!}L^l$, we can determine if $P$ commutes with $H$ by checking if $P$ commutes with $L$. Let $k \in \mathbb{Z}^+$ be the degree of each vertex in $G$, as vertex-transitive graphs are regular \cite{biggs1993algebraic}. We have:
\begin{align*}
PL &= P(D-A) = P(kI-A) = kPI-PA \\
&= kIP-AP = (kI-A)P = (D-A)P = LP   
\end{align*}
Thus, $LP=PL$ and hence $H_tP=PH_t$. Therefore, $h_j= Ph_i$. We conclude that each column of the graph heat kernel is a permutation of every other column of the graph heat kernel, thus preserving column norm.
\end{proof}
Note that vertex transitive graphs are a large class of graphs in algebraic graph theory and contain Cayley graphs, distance transitive graphs, and symmetric graphs \cite{godsil2001algebraic}. For example, the \textit{Shrikande graph} seen in Figure \ref{fig: shrikande} is a Cayley graph on the group $\mathbb{Z}_4 \times \mathbb{Z}_4$ and therefore is also vertex-transitive \cite{godsil2001algebraic}.

\subsection{Strongly Regular Graphs}
Another large class of graphs in algebraic graph theory are the strongly regular graphs. They are an important class of graphs in frame theory and quantum computing due in part to the one-to-one correspondence between \textit{equiangular tight frames} for $\mathbb{R}^N$ that minimize the \textit{coherence} $\max_{i, i'} |\langle \psi_i, \psi_{i'} \rangle|^2$ and a subset of strongly regular graphs \cite{waldron2009construction}. In quantum computing, the equiangular tight frames for $\mathbb{C}^N$ are called SIC-POVMs (\textit{Symmetric Informationally Complete Positive Operator Valued Measurements}) and conditions on their existence are part of the \textit{Zauner conjecture} and relevant for quantum tomography and cryptography \cite{renes2004symmetric}. Further, strongly regular graphs are good candidates for state transfer in continuous-time quantum walks \cite{coutinho2014quantum} \cite{godsil2020state}.

Building on well-understood properties of the spectrum of strongly regular graphs, we show that the graph Gabor frame $\{\psi_{ij}(t)\}_{i,j=1}^N$ is tight and independent of $t \in \mathbb{R}_{\geq 0}$. The spectrum of a strongly regular graph (its adjacency matrix $A \in \mathbb{R}^{N \times N}$) is well-understood in that the eigenvalues and their multiplicities can be explicitly computed in terms of the parameters $(n,k,a,c)$ \cite{biggs1993algebraic}. There are exactly three eigenvalues $\nu_1, \nu_2, \nu_3 \in \mathbb{R}$ of the adjacency matrix $A$:
\begin{equation*}
    \nu_1 = k
\end{equation*}
\begin{equation*}
    \nu_2, \nu_3 = \frac{a - c \pm \sqrt{(a - c)^2 + 4(k-c)}}{2}
\end{equation*}
Since strongly regular graphs are indeed regular, the spectrum  of the graph Laplacian $L=kI-A$ is simply $\{\lambda_i=k-\nu_i\}_{i=1}3$. 

%are given by  have a simple relationship to the eigenvalues $\nu$ of the adjacency matrix $A$. Let $\phi$ be an eigenvector of $A$ with associated eigenvalue $\nu$. Then,
%\begin{equation*}
%\begin{aligned}
%    A\phi = \nu \phi & \iff k\phi - A\phi = 
%    k\phi - \nu \phi \\ & \iff 
 %   (kI-A)\phi = (k-\nu) \phi \\ & \iff 
 %   L\phi = (k-\nu) \phi \\  & \iff
 %   L \phi = \lambda \phi
%\end{aligned}
%\end{equation*}
Let $m_1, m_2, m_3$ denote the multiplicities of the eigenvalues $\lambda_1, \lambda_2, \lambda_3$ of $L$. The graph Laplacian $L$ in our setting has eigenvalue $\lambda_1=0$ with the corresponding normalized eigenvector $\mathbf{1}/\sqrt{N}$ \cite{chung1997spectral}. Since an eigenvalue of the frame operator is the norm-squared of a column of the graph heat kernel by Theorem \ref{thm:eig}, we have the following expression for the eigenvalues of $S(t)$ strongly regular graphs by decomposing the sum in Theorem $\ref{thm:eig}$ according to the eigenvalue multiplicities:
\begin{equation} \label{eq: eigdecomp}
    \begin{split}
        \gamma_i(t) &= \sum_{k=1}^{N} e^{-2\lambda_kt}|\phi_k(v_i)|^2 \\
        &= \frac{1}{N} + e^{-2 \lambda_2 t} \sum_{k=2}^{m_2+1}|\phi_k(v_i)|^2 + e^{-2 \lambda_3 t} \sum_{k=m_2 + 2}^N|\phi_k(v_i)|^2 
    \end{split}
\end{equation}
We leverage this expression along with a simple counting argument to establish that $\{\psi_{ij}(t)\}_{i,j=1}^N$ is tight and independent of $t \in \mathbb{R}_{\geq 0}$.

\begin{theorem} \label{thm: srg} The graph Gabor frame $\{\psi_{ij}(t)\}_{i,j=1}^N$ forms a tight frame for every time $t\in [0, \infty)$ when the underlying graph is strongly regular.
\end{theorem}

\begin{proof}
Since the smallest eigenvalue of $L$ is zero with normalized  eigenvector $\mathbf{1}/\sqrt{N}$, we have for $i \in \{1, \ldots, N\}$:
\begin{equation} \label{eq: eigvec1}
    1 = \frac{1}{N} + \sum_{k=2}^{m_2+1} |\phi_k(v_i)|^2 + \sum_{k=m_2+2}^N |\phi_k(v_i)|^2
\end{equation}
From~\eqref{eq: eigvec1} we get 
%we can solve for the second summation 
%$\sum_{k=m_2+2}^{n}|\phi_k(v_i)|^2$ in the expression for $\gamma_i(t)$ in \ref{eq: eigdecomp}:
\begin{equation} \label{eq: second_sum}
    \begin{aligned}
        \sum_{k=m_2+2}^N |\phi_k(v_i)|^2 &= \frac{N-1}{N}-\sum_{k=2}^{m_2+1} |\phi_k(v_i)|^2
    \end{aligned}
\end{equation}
Substituting~\eqref{eq: second_sum} into~\eqref{eq: eigdecomp}  we see that for any $i\neq j$ 
%and subtracting we arrive at an expression for the graph Gabor frame operator eigenvalue difference:
\begin{equation} \label{eq: eig_diff}
    \gamma_i(t) - \gamma_j(t) = (e^{-2\lambda_2t} - e^{-2\lambda_3 t}) \Big[\sum_{k=2}^{m_2+1}|\phi_k(v_i)|^2-|\phi_k(v_j)|^2 \Big]
\end{equation}
By a simple count, we know that $||Le_i||^2=d^2+d$, since strongly regular graphs are regular. Further, via the spectral theorem the eigenvectors form an orthonormal basis and so $\Phi \Phi^*=I$. In particular,
$$d^2+d =\|Le_i\|^2=\|\Phi \Lambda \Phi^* e_i\|^2= \|\sum_{k=1}^{N} \lambda_k \phi_k \bar{\phi}_k e_i\|^2= \lambda_2^2 \sum_{k=2}^{m_2+1} |\phi_k(v_i)|^2 + \lambda_3^2  \sum_{k=m_2+2}^{N} |\phi_k(v_i)|^2.$$
%We write this out in terms of outer-products:
%\begin{equation*}
  %  \begin{split}
 %   d^2+d &= ||Le_i||^2 \\
%    &= ||\Phi \Lambda \Phi^* e_i||^2 \\
 %   &= ||\sum_{k=1}^{n} \lambda_k \phi_k \bar{\phi}_k e_i||^2 \\
 %   &= ||\sum_{k=2}^{n} \lambda_k \phi_k \bar{\phi}_ke_i||^2 \\
 %   &= ||\lambda_2 \sum_{k=2}^{m_2+1} \phi_k \bar{\phi}_k e_i + \lambda_3 \sum_{k=m_2+2}^{n} \phi_k \bar{\phi}_k e_i||^2 \\
 %   &= |\lambda_2|^2 \sum_{k=2}^{m_2+1} |\phi_k(v_i)|^2 + |\lambda_3|^2  \sum_{k=m_2+2}^{n} |\phi_k(v_i)|^2
 %   \end{split}
%\end{equation*}
It follows that 
%Solving for the summation $\sum_{k=m_2+2}^{n}|\phi_k(v_i)|^2$:
\begin{equation} \label{eq: second_sum_1}
    \sum_{k=m_2+2}^{N} |\phi_k(v_i)|^2 = \frac{(d^2+d) - \lambda_3^2 \sum_{k=2}^{m_2+1} |\phi_k(v_i)|^2}{\lambda_2^2}
\end{equation}
Using~\eqref{eq: second_sum} and~\eqref{eq: second_sum_1} we have 
%and solving for $\sum_{k=2}^{m_2+1} |\phi_k(v_i)|^2$, we get
\begin{equation} \label{eq: partial_sum_eigvec}
    \sum_{k=2}^{m_2+1} |\phi_k(v_i)|^2 = \frac{(\frac{n-1}{n})\lambda_3^2 - (d^2+d)}{\lambda_3^2 - \lambda_2^2}.
\end{equation}
Because \eqref{eq: partial_sum_eigvec} is independent of $i$ and $j$ we conclude that the term $\sum_{k=2}^{m_2+1}|\phi_k(v_i)|^2-|\phi_k(v_j)|^2=0$ in \eqref{eq: eig_diff} is zero. Consequently, 
%in the expression for frame operator eigenvalue difference $\gamma_i(t) - \gamma_j(t)$ in \ref{eq: eigdecomp} is precisely zero. Thus, 
$\gamma_i(t) - \gamma_j(t) = 0$ for all $i\neq j$.
\end{proof}

\section*{Acknowledgement}
This work was partially supported by grants from the National Science Foundation under grant numbers DMS 1814253 and DMS 2205771. 
\bibliography{bib.bib}
\bibliographystyle{abbrv}

\end{document}